\documentclass[11pt,hyp,]{nyjm}
\usepackage{hyperref}

\hypersetup{nesting=true,debug=true,naturalnames=true}
\usepackage{graphicx,amssymb,upref}
\usepackage[capitalize]{cleveref}

\newtheorem{theorem}{Theorem}[section]

\DeclareMathOperator{\Aut}{Aut}
\DeclareMathOperator{\Out}{Out}
\DeclareMathOperator{\Inn}{Inn}

\let\<\langle
\let\>\rangle

\let\uml\"

\title{Short Presentations for $\Aut_n^+$ and $\Out_n^+$} 

\author{Andrea Heald}  
\address{} 
\email{amheald@uw.edu}  


\subjclass[2010]{}

\begin{document} 
 
\begin{abstract}  
This paper uses the methods from \cite{GKKL} to give 'short' presentations for $\Aut^+(F_n)$, the special automorphism group of the free group of rank $n$, and $\Out^+(F_n)$ the special outer automorphism group of the free group of rank $n$. 
\end{abstract} 
\maketitle
\tableofcontents

\section{Introduction}
In this paper we give `short' presentations for $\Aut^+(F_n)$ and $\Out^+(F_n)$.  For this paper the length of a presentation $\langle X | R \rangle$ is defined to be $|X|$ plus the sum of the lengths of all elements of $R$ considered as words in $X$.  A presentation is `short' if the length is $O(\log(n))$.  

This work was inspired by \cite{GKKL} in which Guralnick, Kantor, Kassabov and Lubotzky give short presentations for a number of groups.  In particular they give both a presentation for $S_n$, the symmetric group, of length $O(\log(n)$ and a presentation of $SL_n(F_q)$, the special linear group over the finite field of order $q$, of length $O(\log(n)+\log(q))$.  We achieve our `short' presentations for $\Aut^+(F_n)$ and $\Out^+(F_n)$ by adapting their methods and using the presentation for $\Aut^+(F_n)$ found by Gersten in \cite{Gersten}.  In particular, we combine their `short' presentation of $S_n$ with a presentation for $\Aut^+(F_5)$ via an action of $S_n$ on a free basis of $F_n$.  Since the presentation of $S_n$ is bounded (i.e. has a fixed number of generators and relations), the presentation obtained for $\Aut^+(F_n)$ is also bounded.  However, the presentation obtained for $\Out^+(F_n)$ is not, as it has $O(\log(n))$ generators.  

\section{Preliminaries}
Let $\Aut(F_n)$ be the automorphism group of $F_n$ a free group of rank $n$.  There is a homomorphism $\phi: \Aut(F_n) \rightarrow GL_n(\mathbb{Z})$ obtained from the abelianization of $F_n$.  The determinant of $\Aut(F_n)$ is given by $\det \circ \phi$, and $\Aut^+(F_n)$ is the subgroup of $\Aut(F_n)$ of elements of determinant 1.  In defining $\Out^+$ we note that the group of inner automorphisms, $\Inn(F_n)$, is contained in $\ker(\phi)$ and thus $\Inn(F_n) \subset \Aut^+(F_n)$.  We then define $\Out^+(F_n) := \Aut^+(F_n)/\Inn(F_n)$.

Let $\theta$ be a free basis for $F_n$.  For ease of notation we define $\overline{x} = x^{-1}$ for $x \in \theta$, and  let $E = \{x|x$ or $\overline{x}$ in in $\theta\}$.  The Nielsen map $E_{ab}$ is given by $E_{ab}(a) = ab$ and for $c \in E$ with $c \neq a, \overline{a}$, $E_{ab}(c)=c$.  Clearly this is in $\Aut^+(F_n)$ (notice that it maps to an elementary matrix under $\phi$.)
Gersten found a presentation of $\Aut^+(F_n)$ in terms of Nielsen maps:
\begin{theorem} \cite{Gersten}
If $n \geq 3$ a presentation for $\Aut^+(F_n)$ is given by: \\ 
Generators: $E_{ab}$ with $a \neq b, \overline{b}$, $a,b \in E$.
\begin{enumerate}
\item $E_{ab}^{-1} = E_{a\overline{b}}$
\item $[E_{ab},E_{cd}] = 1$ if $a \neq c, d, \overline{d}$, $b \neq c, \overline{c}$
\item $[E_{ab},E_{bc}] = E_{ac}$ if $a \neq c,\overline{c}$.
\item $w_{ab} = w_{\overline{a}\overline{b}}$ where $w_{ab} = E_{ba}E_{\overline{a}b}E_{\overline{b}\overline{a}}$.
\item $w_{ab}^4 = 1$.
\end{enumerate}
\end{theorem}
With the convention that maps are composed left to right, and $[x,y] = xyx^{-1}y^{-1}$.  The length of this presentation is $O(n^2)$.

By considering the analogous case of elementary matrices in $SL_n(\mathbb{Z})$ we construct a presentation of $\Aut^+(F_n)$ of length $O(n)$.
\begin{theorem}\label{longpres}
If $n \geq 3$ a presentation for $\Aut^+(F_n)$ is given by: \\
Generators: $E_{ba_1}, E_{b\overline{a_1}},E_{a_1b},E_{\overline{a_1}b}$ where $b \in E, b \neq a_1, \overline{a_1}$. \\
Relations:
\begin{enumerate}
\item $E_{ba_1}^{-1} = E_{b\overline{a_1}}$, $E_{a_1b}^{-1} = E_{a_1\overline{b}}$
\item $[E_{ba_1},E_{ca_1}] = 1$, $b \neq c, \overline{c}$
\item $[E_{a_1b},E_{\overline{a_1}c}] = [E_{\overline{a_1}b},E_{a_1c}] = 1$.
\item $[E_{b\overline{a_1}},E_{\overline{a_1}c}] = [E_{ba_1},E_{a_1c}]$ $b \neq c, \overline{c}$ 
\item $[[E_{ba_1},E_{a_1c}],E_{da_1}] = [[E_{ba_1},E_{a_1c}],E_{d\overline{a_1}}] = 1$, $b \neq d, c, \overline{c}$, $c \neq d, \overline{d}$
\item $[[E_{ba_1},E_{a_1c}],E_{a_1d}] = [[E_{ba_1},E_{a_1c}],E_{\overline{a_1}d}] = 1$, $b \neq d, \overline{d}, c, \overline{c}$
\item $[E_{a_1d},[E_{ba_1},E_{a_1c}]] = [E_{\overline{a_1}d},[E_{ba_1},E_{a_1c}]] = 1$, $b \neq d, \overline{d}, c, \overline{c}$
\item $[[E_{ba_1},E_{a_1c}],E_{ca_1}] = E_{ba_1}$, $b \neq c,\overline{c}$
\item $[E_{ba_1},E_{a_1c}],E_{c\overline{a_1}}] = E_{b\overline{a_1}}$, $b \neq c, \overline{c}$
\item $[E_{a_1b},[E_{ba_1},E_{a_1c}]] = E_{a_1c}$, $b \neq c, \overline{c}$
\item $[E_{\overline{a_1}b},[E_{ba_1},E_{a_1c}]] = E_{\overline{a_1}c}$, $b \neq c, \overline{c}$.
\item $w_{a_1b} = w_{\overline{a_1}\overline{b}}$, $w_{ba_1} = w_{\overline{b}\overline{a_1}}$ where $w_{ab}$ is defined as in the previous theorem.
\item $(w_{a_1b})^4 = (w_{ba_1})^4 = 1$
\end{enumerate}
\end{theorem}

\begin{proof}
Set $E_{bc} = [E_{ba_1},E_{a_1c}]$.  We can see $E_{bc}$ will satisfy relations (1), (2) and (3) in Gersten's presentation.  It remains to show that relations (4) and (5) also hold.

Notice that $E_{bc} = w_{a_1b}^{-1}E_{\overline{a_1}c}w_{a_1b}$ and $E_{cb} = w_{a_1b}^{-1}E_{c\overline{a_1}}w_{a_1b}$.  Then \begin{eqnarray*} w_{bc} &:=& E_{cb}E_{\overline{b}c}E_{\overline{c}\overline{b}} 
\\ &=& w_{a_1b}^{-1}E_{c\overline{a_1}}E_{a_1c}E_{\overline{c}a_1}w_{a_1b} 
\\ &=& w_{a_1b}^{-1}w_{\overline{a_1}c}w_{a_1b}
\\ &=& w_{a_1b}^{-1}w_{a_1\overline{c}}w_{a_1b}
\\ &=& w_{\overline{b}\overline{c}}
\end{eqnarray*}
as desired.  The same trick shows that $w_{bc}^4 = 1$, and thus this is a presentation for $\Aut^+(F_n)$.
\end{proof}

\section{Proof of Result}
We now construct our short presentation of $\Aut^+(F_n)$.  The basic idea is to combine the previous presentation of $\Aut^+(F_5)$ with a short presentation of the symmetric group, $S_{n-1}$, viewed as automorphisms that permute the elements of the free basis of $F_n$ (excluding $a_1$) and either fix $a_1$ or send it to $\overline{a_1}$ (depending on which yields an element of determinant 1.)

The desired `short' presentation of $S_{n-1}$ comes from the proof of Theorem 6.1 in \cite{GKKL} and has some added useful requirements.
\begin{theorem}\label{GKSn}\cite{GKKL}
There exists a bounded presentation $\langle X|R \rangle$ of $S_{n-1}$ (viewed as acting on $2,\ldots,n$) of length $O(\log(n)$ such that the following hold:
\begin{enumerate}
\item There exists $X_1,X_2 \subseteq X$ such that $\langle X_1 \cup X_2 \rangle$ projects onto the stabilizer of 2 in $S_{n-1}$.
\item $X_1$ projects into $A_{n-1}$ and $X_2 \neq 0$ projects outside $A_{n-1}$ 
\item The elements $(2,3)$, $(2,3,4,5)$ and $(2,3,\ldots,n)$ have length $O(\log(n))$ in the generators.
\end{enumerate}
\end{theorem}

Using this we can now prove the main theorem:

\begin{theorem}\label{SAn}
The group $\Aut^+(F_n)$ has a bounded presentation of length $O(\log(n))$.
\end{theorem}
\begin{proof}
We can assume $n \geq 6$. Let $E = \{a_1,\ldots,a_n,\overline{a_1},\ldots,\overline{a_n}\}$, $\langle X | R \rangle = S_{n-1}$ as in the previous theorem and $\Aut^+(F_5) = \langle \tilde{X}|\tilde{R} \rangle$ (viewed as acting on $a_1,\ldots,a_5$) be the presentation given in \cref{longpres}.
Then we have a presentation of $\Aut^+(F_n)$ given by: \\
Generators: $X \cup \tilde{X}$ \\
Relations: 
\begin{enumerate}
\item $R \cup \tilde{R}$
\item $E_{a_1a_2}^{x_1} = E_{a_1a_2}$, $E_{a_2a_1}^{x_1} = E_{a_2a_1}$, $E_{\overline{a_1}a_2}^{x_1} = E_{\overline{a_1}a_2}$, $E_{\overline{a_2}a_1}^{x_1} = E_{\overline{a_2}a_1}$ for all $x_1 \in X_1$
\item $E_{a_1a_2}^{x_2} = E_{\overline{a_1}a_2}$, $E_{a_2a_1}^{x_2} = E_{a_2\overline{a_1}}$, $E_{\overline{a_1}a_2}^{x_2} = E_{a_1a_2}$, $E_{\overline{a_2}a_1}^{x_2} = E_{\overline{a_2}\overline{a_1}}$ for all $x_2 \in X_2$
\item $(2,3) = w_{a_1a_2}^2w_{a_2a_3}$
\item $(2,3,4,5) = w_{a_1a_4}^2w_{a_4a_5}w_{a_1a_3}^2w_{a_3a_4}w_{a_1a_2}^2w_{a_2a_3}$
\end{enumerate}

Clearly this presentation has the desired length, and  we can see it has $|X|+5$ generators and $|R|+4|X_1|+4|X_2|+1102$ relations, and thus is a bounded presentation.  

We now show that this is a presentation for $\Aut^+(F_n)$.  For $k \geq 6$, set $E_{a_1a_k} = (2,k)E_{\overline{a_1}a_2}(2,k)$, $E_{\overline{a_1}a_k} = (2,k)E_{a_1a_2}(2,k)$, $E_{a_ka_1} = (2,k)E_{a_2\overline{a_1}}(2,k)$, $E_{\overline{a_k}a_1} = (2,k)E_{\overline{a_2}\overline{a_1}}(2,k)$.  I claim that these elements satisfy all the relations in \cref{longpres}.

We observe the following two facts about $E_{a_ia_j} = [E_{a_ia_1},E_{a_1a_j}]$ (where $i,j \geq 6$):
First \begin{eqnarray*}&&[E_{a_ia_1},E_{a_1a_j}] \\
 &=& (2,i)E_{a_2\overline{a_1}}(2,i)(2,j)E_{\overline{a_1}a_2}(2,j)(2,i)E_{a_2a_1}(2,i)(2,j)E_{\overline{a_1}\overline{a_2}}(2,j) \\
&=& (2,i)E_{a_2\overline{a_1}}(3,i,j)(2,3)(3,i)E_{\overline{a_1}a_2}(3,i)(2,3)(3,j,i)E_{a_2a_1}(3,i,j)(2,3)(3,i)E_{\overline{a_1}\overline{a_2}}(2,j) \\
&=& (2,i)(3,i,j)E_{a_2\overline{a_1}}w_{a_1a_2}^2w_{a_2a_3}E_{a_1a_2}w_{a_1a_2}^2w_{a_2a_3}E_{a_2a_1}w_{a_1a_2}^2w_{a_2a_3}E_{a_1\overline{a_2}}(3,i)(2,j) \\
&=& (2,i)(3,i,j)(2,3)E_{a_3a_1}E_{a_1a_2}E_{a_3\overline{a_1}}E_{a_1\overline{a_2}}(3,i)(2,j) \\
&=& (2,j)(3,i)E_{a_3a_2}(3,i)(2,j) \\
&=& (2,i)(3,i,j)(2,3)E_{a_3\overline{a_1}}E_{\overline{a_1}a_2}E_{a_3a_1}E_{\overline{a_1}\overline{a_2}}(3,i)(2,j) \\
&=& [E_{a_i\overline{a_1}},E_{\overline{a_1}a_j}]
\end{eqnarray*}
And second
\begin{eqnarray*}
(2,3,4,5)E_{a_1a_k}(2,5,4,3) &=& (2,3,4,5)(2,k)E_{\overline{a_1}a_2}(2,k)(2,5,4,3) \\
&=& (2,k)(k,3,4,5)E_{\overline{a_1}a_2}(k,5,4,3)(2,k) \\
&=& (2,k)E_{a_1a_2}(2,k) \\
&=& E_{\overline{a_1}a_k}.
\end{eqnarray*}
We can see $(2,3,4,5)$ behaves similarly with respect to $E_{\overline{a_1}a_k}$, $E_{a_ka_1}$, $E_{\overline{a_k}a_1}$. Thus $(2,3,4,5)$ commutes with $[E_{a_ia_1},E_{a_1a_j}]$.  Combining these we have \begin{eqnarray*} E_{a_ia_j} &=& E_{a_ia_j}^{(2,3,4,5)^2} \\ &=& (4,j)(5,i)E_{a_5a_4}(5,i)(4,j).\end{eqnarray*}  
These facts can be used to establish the more complicated relations.  For example if $a_i \neq a_j,\overline{a_j},a_k,\overline{a_k}$ we have that:
\begin{eqnarray*}
&&[[E_{a_ia_1},E_{a_1a_j}],E_{a_1a_k}] \\ 
&=& E_{a_ia_j}E_{a_1a_k}E_{a_ia_j}^{-1}E_{a_1a_k}^{-1} \\
&=& E_{a_ia_j}(2,5,4,3)E_{\overline{a_1}a_k}(2,3,4,5)E_{a_ia_j}^{-1}E_{a_1a_k}^{-1} \\
&=& E_{a_ia_j}(2,5,4,3)(2,k)E_{a_1a_2}(2,k)(2,3,4,5)E_{a_ia_j}^{-1}E_{a_1a_k}^{-1} \\
&=& E_{a_ia_j}(3,k)(2,5,4,3)E_{a_1a_2}(2,3,4,5)(3,k)E_{a_ia_j}^{-1}E_{a_1a_k}^{-1} \\
&=& E_{a_ia_j}(3,k)E_{\overline{a_1}a_3}(3,k)E_{a_ia_j}^{-1}(3,k)E_{\overline{a_1}a_3}^{-1}(3,k) \\
&=& (3,k)E_{a_ia_j}E_{\overline{a_1}a_3}E_{a_ia_j}^{-1}E_{\overline{a_1}a_3}^{-1}(3,k) \\
&=& (3,k)(2,5,4,3)E_{a_ia_j}(2,3,4,5)E_{\overline{a_1}a_3}(2,5,4,3)E_{a_ia_j}^{-1}(2,3,4,5)E_{\overline{a_1}a_3}^{-1}(3,k) \\
&=& (3,k)(2,5,4,3)E_{a_ia_j}E_{a_1a_2}E_{a_ia_j}^{-1}E_{a_1a_2}^{-1}(2,3,4,5)(3,k) \\
&=& (3,k)(2,5,4,3)(4,j)(5,i)E_{a_5a_4}E_{a_1a_2}E_{a_5a_4}^{-1}E_{a_1a_2}^{-1}(5,i)(4,j)(2,3,4,5)(3,k) \\
&=& 1
\end{eqnarray*}
The other computations for the other needed relations are similar.  Thus we have the desired short presentation for $\Aut^+(F_n)$.

\end{proof}

We now use this to obtain our presentation for $\Out^+(F_n)$ by constructing an inner automorphism in terms of the $E_{ab}$s and observing that there is generating set of $\Inn(F_n)$ consisting of conjugate elements.

\begin{theorem}
The group $\Out^+(F_n)$ has a presentation of length $O(\log(n))$.
\end{theorem}

\begin{proof}
Let $\langle X| R \rangle$ be the presentation obtained in \cref{SAn}.  If $n$ is even set $\sigma = (2,3,\ldots,n)$, $\xi = E_{a_2a_1}E_{\overline{a_2}a_1}$, and if $n$ is odd set $\sigma = (2,3,\ldots,n)(2,3) = (3,\ldots,n)$ and $\xi = E_{a_3a_1}E_{\overline{a_3}a_1}$.  Then there is a presentation for $\Out^+$ given by: 
\\
Generators: $X \cup \{\gamma\}$
\\
Relations: 
\begin{enumerate}
\item $R$
\item $\gamma = \xi\sigma$
\item $\gamma^{n-1} = 1$ if $n$ is even and $E_{a_2a_1}E_{\overline{a_2}a_1}\gamma^{n-2} = 1$ if $n$ is odd.
\end{enumerate}
Notice that the length of this presentation is not $O(\log(n))$, however by using the variation of Horner's Rule used in \cite{GKKL} Proposition 3.1 (3.3) we can convert this to a `short' presentation.  To do this we add the generators $\gamma_1 = \gamma^4,\gamma_2 = \gamma^{4^2},\ldots,\gamma_m = \gamma^{4^m}$ where $m = \lfloor{\log_4(n-2)}\rfloor$ and the relations $\gamma_i = (\gamma_{i-1})^4$.  Then $n-2 = a_0 + 4a_1 + \cdots + 4^ma_m$ with $0 \leq a_1,\ldots,a_m \leq 4$ and thus $\gamma^{n-2}$ can be replaced with $\gamma^{a_0}\gamma_1^{a_1}\cdots\gamma_m^{a_m}$, yeilding a presentaiton with length $O(\log(n))$.

To show that the original presentation is in fact $\Out^+(F_n)$ we show that $\Inn(F_n)$ is generated by words of the form $w_{a_1a_i}\gamma^{n-1}w_{a_1a_i}^{-1}$ if $n$ is even and $w_{a_1a_i}E_{a_2a_1}E_{\overline{a_2}a_1}\gamma^{n-2}w_{a_1a_i}^{-1}$ if $n$ is odd.

The map $f_1:F_n \rightarrow F_n$ given by $f_1(x) = a_1^{-1}xa_1$ can be expressed as $f_1 = \gamma^{n-1}$ or $E_{a_2a_1}E_{\overline{a_2}a_1}\gamma^{n-2}$ because \begin{eqnarray*} f_1 &=& E_{a_2a_1}E_{\overline{a_2}a_1}E_{a_3a_1}E_{\overline{a_3}a_1}\cdots E_{a_na_1}E_{\overline{a_n}a_1} \\ &=& E_{a_2a_1}E_{\overline{a_2}a_1}(\sigma E_{a_2a_1}E_{\overline{a_2}a_1} \sigma^{-1})(\sigma^2 E_{a_2a_1}E_{\overline{a_2}a_1} \sigma^{-2})\cdots(\sigma^{n-2}E_{a_2a_1}E_{\overline{a_2}a_1}\sigma^{2-n}) \\  &=& \gamma^{n-1}\end{eqnarray*} if $n$ is even.  The analogous result holds for $n$ odd.  

It is clear that $\Inn(F_n)$ is generated by the maps $f_i:F_n \rightarrow F_n$ given by $f_i(x) = a_i^{-1}xa_i$.  To get from $f_1$ to $f_i$ we observe that $w_{a_ia_1}E_{a_ja_1}E_{\overline{a_j}a_1}w^{-1}_{a_ia_1} = (E_{a_ja_i}E_{\overline{a_j}a_i})^{-1}$ when $j \neq i$ and $(E_{a_1a_i}E_{\overline{a_1}a_i})^{-1}$ when $j = i$.  Thus 
$w_{a_ia_1}f_1^{-1}w_{a_ia_1}^{-1} = f_i$ and we have that the map from $\Aut^+(F_n)$ to this presentation is surjective and has kernel equal to $\Inn(F_n)$.  Thus this is a presentation of $\Out^+(F_n)$.

\end{proof}

\bibliographystyle{amsplain}

\begin{thebibliography}{10}
\bibitem{Gersten}
Gersten, S.M.
\newblock A presentation for the special automorphism group of a free
              group.
\newblock {\em J. Pure Appl. Algebra}, 33(3):269--279, 1984.


\bibitem{GKKL}
Guralnick, R. M. and Kantor, W. M. and Kassabov, M. and
             Lubotzky, A.
\newblock Presentations of finite simple groups: a quantitative
              approach.
\newblock {\em Journal of the American Mathematical Society}, 21(3):711--774, 2008.



\end{thebibliography}

\end{document}